\documentclass[11pt]{article}

\usepackage[T1]{fontenc}
\usepackage[utf8]{inputenc}
\usepackage{lmodern}
\usepackage{microtype}
\usepackage{amsmath,amssymb,amsthm}
\usepackage{geometry}
\usepackage{hyperref}
\usepackage{aliascnt}
\usepackage[nameinlink,capitalise,noabbrev]{cleveref}

\hypersetup{hidelinks}
\numberwithin{equation}{section}

\newtheorem{theorem}{Theorem}[section]
\newaliascnt{proposition}{theorem}
\newtheorem{proposition}[proposition]{Proposition}
\aliascntresetthe{proposition}
\newaliascnt{lemma}{theorem}
\newtheorem{lemma}[lemma]{Lemma}
\aliascntresetthe{lemma}
\newaliascnt{corollary}{theorem}
\newtheorem{corollary}[corollary]{Corollary}
\aliascntresetthe{corollary}
\theoremstyle{remark}

\theoremstyle{plain}

\newcommand{\Sym}{\operatorname{Sym}}
\newcommand{\Fix}{\operatorname{Fix}}
\newcommand{\trc}{\operatorname{tr}}
\newcommand{\normal}[1]{\langle\!\langle #1\rangle\!\rangle}
\newcommand{\eps}{\varepsilon}

\title{Soficity of One-Relator Groups and Reducible Presentations}
\author{Jinhyuk Son}
\date{October 2026}

\begin{document}
\maketitle

\begin{abstract}
We prove that every one-relator group is sofic, answering a question of Nate Brown. More generally, every group admitting a reducible presentation without proper powers is sofic, with one proper-power relation allowed at the final reducible step. The proof starts from the endpoint-preserving edge replacements of Poulin--Wr\'obel. We retain the free-group word carried by each terminal replacement and use it to define an $F(S)$-valued cocycle repair; in its relative form, the repair changes only the coordinate of the generator added at that step and leaves all previously chosen coordinates unchanged. Coinduction transports the lower cocycle to the next one-relator-product stage, so this repair iterates along reducible presentations without increasing the defects of earlier relators. For a final relator $w^m$, cyclic translates of the repaired set amplify a zero-defect set of measure close to $1/m$ to one of measure close to one. A cocycle criterion then converts arbitrarily small relator defect into finite permutation approximations by realizing the relevant fixed-point statistics in a treeable skew-product orbit relation.
\end{abstract}


\section{Introduction}

Sofic groups are groups admitting finite approximation by permutations. The notion goes back to Gromov \cite{Gromov1999} and Weiss \cite{Weiss2000}; basic approximation and closure results are due to Elek--Szab\'o \cite{ElekSzabo2005,ElekSzabo2006}, and measured formulations through actions and equivalence relations were developed by Elek--Lippner \cite{ElekLippner2010} and P\u{a}unescu \cite{Paunescu2011}. Every amenable group and every residually finite group is sofic, but recent constructions of non-sofic groups, including torsion-free examples, show that the class is not universal \cite{OpenAI2026,FournierFacio2026,KunThom2026}.

One-relator groups have been a central testing ground in combinatorial group theory since Magnus's Freiheitssatz \cite{Magnus1930}; see Lyndon--Schupp \cite{LyndonSchupp1977} and the recent survey of Linton--Nyberg-Brodda \cite{LintonNybergBrodda2025}. If the relator is a proper power, the group has torsion and is residually finite by Wise's virtual-specialness theory \cite{Wise2021}, hence sofic. Torsion-free one-relator groups need not be residually finite \cite{Meskin1972}. Some non-residually finite examples were shown to be sofic using HNN decompositions together with permanence over amenable subgroups \cite{CollinsDykema2011,BannonNoblett2010}. Bannon also established hyperlinearity for a related non-residually solvable example \cite{Bannon2011}, and Berlai later obtained further sofic families by a different argument \cite{Berlai2026}. In 2008 Pestov recorded a question, attributed to Nate Brown, asking whether every one-relator group is sofic \cite[Open Question~4.10]{Pestov2008}. The question was subsequently discussed by Bannon--Noblett \cite{BannonNoblett2010} and more recently by Berlai \cite{Berlai2026}. We answer it affirmatively.

\begin{theorem}\label{thm:main}
Every one-relator group is sofic.
\end{theorem}

The proof separates an approximation criterion from a group-theoretic repair. Write a nontrivial cyclically reduced relator as $r=w^m$, where $w$ is not a proper power. Poulin--Wr\'obel \cite{PoulinWrobel2026} construct free probability-measure-preserving (pmp) actions in which selected generator edges are recursively replaced by complementary relator paths while preserving their endpoints. Endpoint preservation does not by itself control the free-group word carried by a replacement path. We therefore retain that word in an $F(S)$-valued cocycle lifting the quotient action and choose the coordinate of a selected generator $s$ from the terminal replacement. This forces the relator cocycle to be trivial on the repaired set.

In the iterated setting, the repair is relative: coordinates already chosen for the lower generators are left unchanged. Combined with coinduction, this gives a one-step procedure that transports a lower cocycle through a one-relator-product extension, preserves the lower relator defects, and makes the new relator defect small. For $w^m$, the repaired set has measure close to $1/m$ and its cyclic translates amplify the zero-defect region to measure close to one. The cocycle criterion then realizes the relevant fixed-point statistics in a treeable skew-product orbit relation and converts them into finite permutation models.

The same relative step can be iterated along reducible presentations in the sense of Howie \cite{Howie1982}.

\begin{theorem}\label{thm:reducible}
Every group admitting a reducible presentation without proper powers is sofic.
\end{theorem}

For finite presentations the proof proceeds one relation at a time. The lower group is locally indicable at each stage, and the relative repair changes only the coordinate of the new generator. A finite-subpresentation argument gives the arbitrary case. The same method permits a proper-power relation at the final step.

\begin{theorem}\label{thm:final-power}
Let $\Lambda=\langle S'\mid R'\rangle$ be given by a finite reducible presentation without proper powers. Let $s\notin S'$ and let $u\in F(S'\sqcup\{s\})$ have image in $\Lambda*\langle s\rangle$ which is not conjugate into $\Lambda$ and is not a proper power. Then, for every $m\ge1$,
\[
 (\Lambda*\langle s\rangle)/\normal{u^m}
\]
is sofic.
\end{theorem}

The paper is organized around the two parts of the proof described above. \Cref{sec:criterion} gives the cocycle criterion converting small relator defect into finite permutation models. \Cref{sec:repair} isolates the terminal cocycle repair, and \cref{sec:pwinput} records the edge-cutting input from Poulin--Wr\'obel in the finite alphabet needed here. \Cref{sec:one-step} combines these ingredients with coinduction into the relative one-step repair. The main one-relator theorem is then proved in \cref{sec:main-proof}, and \cref{sec:reducible-proof,sec:consequences} carry the construction through finite and arbitrary reducible presentations.

\section{Preliminaries}

Let $S$ be finite and let $F(S)$ be the free group on $S$, with identity $1_{F(S)}$. We write words as
\[
 u=a_k\cdots a_1,\qquad a_i\in S^{\pm1},
\]
so that, for a left action, $a_1$ is traversed first. All path and cocycle formulas use this convention. For a subset $B$ of a group, we write $\normal{B}$ for its normal closure in the ambient group when that group is clear.

For a self-map $f$ of a set, write $\Fix(f)$ for its fixed-point set. For $d\ge1$, write $\Sym(d)$ for the symmetric group on $\{1,\ldots,d\}$. For $\sigma,\sigma'\in\Sym(d)$ put
\[
 d_{\mathrm H}(\sigma,\sigma')=\frac1d|\{j:\sigma(j)\ne\sigma'(j)\}|,
 \qquad
 \trc_d(\sigma)=\frac{|\Fix(\sigma)|}{d}.
\]
A group $\Gamma$ is sofic if for every finite $K\subseteq\Gamma$ and every $\delta>0$ there are $d\ge1$ and a map $\phi:K\to\Sym(d)$ with $\phi(1)=1$ such that
\[
 d_{\mathrm H}(\phi(gh),\phi(g)\phi(h))<\delta
\]
whenever $g,h,gh\in K$, and
\[
 d_{\mathrm H}(\phi(g),\phi(h))>1-\delta
 \qquad(g\ne h,\ g,h\in K).
\]
Soficity is local: a group is sofic if and only if each of its finitely generated subgroups is sofic, as follows directly from the finite-set definition. We use the following sufficient finite-word criterion; see \cite{ElekSzabo2005,ElekSzabo2006,Pestov2008}.

If $q:F(S)\twoheadrightarrow\Gamma$ is a quotient map, it is enough to produce, for every finite $W\subseteq F(S)$ and every $\delta>0$, some $d\ge1$ and a homomorphism $\sigma:F(S)\to\Sym(d)$ such that
\[
 q(u)=1\Rightarrow d_{\mathrm H}(\sigma(u),1)<\delta,
 \qquad
 q(u)\ne1\Rightarrow d_{\mathrm H}(\sigma(u),1)>1-\delta
 \quad(u\in W).
\]
Indeed, choose words representing a prescribed finite subset of $\Gamma$ and the products appearing in the finite test, and apply bi-invariance of normalized Hamming distance.

All probability spaces are standard Borel probability spaces. After restricting an essentially free action to the invariant conull Borel set of points with trivial stabilizer, we regard a free pmp action as pointwise free.

\subsection*{One-relator products and reducible presentations}

If $\Lambda$ and $\Delta$ are groups and $w\in\Lambda*\Delta$, write
\[
 \langle\Lambda,\Delta\mid w\rangle=(\Lambda*\Delta)/\normal{w}.
\]
Following Howie, this is a \emph{one-relator product} when $w$ is not conjugate in $\Lambda*\Delta$ into either factor. For a one-relator product of locally indicable groups, the canonical maps of the factors into the quotient are injective \cite{Howie1981}. Moreover, a one-relator product of locally indicable groups is locally indicable when its defining element is not a proper power \cite{Howie1982}. We use these results in the form summarized by Poulin--Wr\'obel \cite[Section~2.4]{PoulinWrobel2026}.

We recall the notion of reducibility used in \cite{Howie1982,JaikinLinton2025,PoulinWrobel2026}. A presentation is \emph{reducible} if every finite subpresentation can be written
\[
 \langle s_1,\ldots,s_k,t\mid w_1,\ldots,w_\ell,u\rangle,
\]
where each $w_i$ is a word in $s_1,\ldots,s_k$ and, with
\[
 \Lambda=\langle s_1,\ldots,s_k\mid w_1,\ldots,w_\ell\rangle,
\]
the element represented by $u$ in $\Lambda*\langle t\rangle$ is not conjugate into $\Lambda$. If $R$ is the defining-relator set and $r\in R$, put
\[
 \Gamma_r=\langle S\mid R\setminus\{r\}\rangle.
\]
The presentation has \emph{no proper powers} when no $r\in R$ represents a proper power in $\Gamma_r$. Reducibility is inherited by subpresentations, and the no-proper-power condition is inherited as well \cite[Section~4]{PoulinWrobel2026}. For a finite reducible presentation, one may remove a generator and one relation so that the remaining presentation is again reducible.

\section{A cocycle criterion for soficity}\label{sec:criterion}

Let $S$ be finite, let $R\subseteq F(S)$, and put
\[
 \Gamma=\langle S\mid R\rangle,\qquad q:F(S)\twoheadrightarrow\Gamma.
\]
Let $\Gamma\curvearrowright(X,\mu)$ be a free pmp action. We pull the action back to $F(S)$ through $q$.

We consider Borel $1$-cocycles for this action with values in the discrete group $F(S)$, that is, Borel maps
\[
 \alpha:F(S)\times X\longrightarrow F(S)
\]
satisfying
\[
 \alpha(uv,x)=\alpha(u,v\cdot x)\alpha(v,x)
 \qquad(u,v\in F(S),\ x\in X).
\]
We impose in addition the lifting condition
\begin{equation}\label{eq:lift}
 q(\alpha(u,x))=q(u)\qquad(u\in F(S),\ x\in X).
\end{equation}
Since the $\Gamma$-action is free, \eqref{eq:lift} is equivalent to
\begin{equation}\label{eq:endpoint}
 \alpha(u,x)\cdot x=u\cdot x.
\end{equation}
For $a\in S^{\pm1}$ put $p_a(x)=\alpha(a,x)$. Conversely, Borel maps $p_a:X\to F(S)$ satisfying
\begin{equation}\label{eq:generator-compatibility}
 p_a(x)\cdot x=a\cdot x,
 \qquad
 p_{a^{-1}}(a\cdot x)=p_a(x)^{-1}
\end{equation}
for all $a\in S$ and $x\in X$ determine such a cocycle uniquely, because $F(S)$ is free. For $u=a_k\cdots a_1$ we write
\[
 P_u(x)=p_{a_k}(a_{k-1}\cdots a_1\cdot x)\cdots p_{a_1}(x)=\alpha(u,x).
\]

For $u\in F(S)$ set
\[
 \operatorname{err}_\alpha(u)=\mu\{x:P_u(x)\ne1_{F(S)}\}.
\]
The choice $p_a\equiv a$ always satisfies \eqref{eq:generator-compatibility}, but for a nontrivial relator $r\in R$ it gives $P_r(x)=r$ everywhere and hence $\operatorname{err}_\alpha(r)=1$. Thus the canonical cocycle retains the word $u$ itself as an element of $F(S)$, even when $u$ is trivial in $\Gamma$. The repair in \cref{sec:repair} instead changes the free-group representative of selected generator moves while preserving their endpoints, so that the $F(S)$-value of a relator loop can become trivial on a large set.
For $u,v\in F(S)$, the cocycle identity and invariance of $\mu$ give
\begin{equation}\label{eq:error-basic}
 \operatorname{err}_\alpha(uv)\le\operatorname{err}_\alpha(u)+\operatorname{err}_\alpha(v),
 \qquad
 \operatorname{err}_\alpha(u^{-1})=\operatorname{err}_\alpha(u).
\end{equation}
If $k\in F(S)$ acts trivially on $X$, then
\begin{equation}\label{eq:error-conj}
 \operatorname{err}_\alpha(gkg^{-1})=\operatorname{err}_\alpha(k)\qquad(g\in F(S)).
\end{equation}
Indeed, $P_{gkg^{-1}}(x)$ is a conjugate of $P_k(g^{-1}\cdot x)$.

We will also use the corresponding pointwise observation. If $k$ acts trivially and $gk=kg$ in $F(S)$, then the cocycle identity gives
\begin{equation}\label{eq:zero-set-invariance}
 P_k(g\cdot x)=P_g(x)P_k(x)P_g(x)^{-1}.
\end{equation}
Thus the set $\{x:P_k(x)=1_{F(S)}\}$ is $g$-invariant.

We use below the Elek--Lippner notion of sofic action, which is formulated for actions of the countably generated free group $F_\infty$. This differs from P\u{a}unescu's notion of sofic action \cite[Section~1.1]{Paunescu2011}, although the corresponding notions of sofic equivalence relation agree \cite[Propositions~4.5--4.6]{Paunescu2011}.

The following lemma realizes the required fixed-point statistics by finite permutation actions.

\begin{lemma}\label{lem:trace-realization}
For every finite $U\subseteq F(S)$ and every $\eta>0$, there are $d\ge1$ and a homomorphism
\[
 \sigma:F(S)\longrightarrow\Sym(d)
\]
such that
\[
 \left|\trc_d(\sigma(u))-\mu\{x:P_u(x)=1_{F(S)}\}\right|<\eta
 \qquad(u\in U).
\]
\end{lemma}

\begin{proof}
Choose a free atomless pmp action $F(S)\curvearrowright(Y,\nu)$, for instance a Bernoulli action, and define
\begin{equation}
 \rho(u)(x,y)=\bigl(u\cdot x,\,P_u(x)\cdot y\bigr).
\end{equation}
For the canonical choice $p_a\equiv a$ discussed above, this is just the diagonal action; the repair changes these fixed-point statistics for relators. The cocycle identity makes $\rho$ an $F(S)$-action. It is pmp as well: for fixed $u$, partition $X$ into the countably many Borel sets $X_g=\{x:P_u(x)=g\}$; on $X_g\times Y$ the map is $(x,y)\mapsto(u\cdot x,g\cdot y)$. Each such piece preserves measure and the images $u\cdot X_g$ are pairwise disjoint, so summing over $g\in F(S)$ preserves $\mu\times\nu$.

Let $\Delta$ be the diagonal action $\Delta(g)(x,y)=(g\cdot x,g\cdot y)$. Write $E_\rho$ and $E_\Delta$ for the orbit equivalence relations of $\rho$ and $\Delta$, respectively. The action $\Delta$ is free because the action on $Y$ is free, so its usual Cayley graphing is a treeing of $E_\Delta$. The endpoint identity \eqref{eq:endpoint} gives
\[
 \rho(u)(x,y)=\Delta(P_u(x))(x,y),
\]
and hence $E_\rho\subseteq E_\Delta$. By Gaboriau's theorem that measured subequivalence relations of treeable relations are treeable \cite[Th\'eor\`eme~IV.4]{Gaboriau2000}, $E_\rho$ is treeable.

Enumerate $S=\{a_1,\ldots,a_{|S|}\}$, regard these as the first $|S|$ generators of $F_\infty$, and extend $\rho$ to an action $\widetilde\rho$ of $F_\infty$ by letting all remaining generators act trivially. Then $E_{\widetilde\rho}=E_\rho$.

Since $E_\rho$ is treeable, Elek--Lippner's Theorem~4 implies that $E_\rho$ is a sofic equivalence relation. By their definition, there is therefore a sofic $F_\infty$-action $\beta$ whose orbit relation is $E_\rho$. The actions $\beta$ and $\widetilde\rho$ are measured orbit equivalent, so Elek--Lippner's Theorem~1 implies that $\widetilde\rho$ is sofic. Definition~1.3 then gives finite $F_\infty$-models whose rooted labelled neighbourhood distributions converge to those of $\widetilde\rho$ \cite[Definition~1.3 and Theorems~1,~4]{ElekLippner2010}.

For the fixed finite set $U$, write $|u|_S$ for reduced word length in the basis $S$, and take
\[
 \operatorname{rad}(U):=\max\bigl(\{|S|\}\cup\{|u|_S:u\in U\}\cup\{1\}\bigr).
\]
In Elek--Lippner's convention, a rooted radius-$r$ neighbourhood uses only the first $r$ generators of $F_\infty$. Thus an Elek--Lippner rooted $\operatorname{rad}(U)$-neighbourhood displays every generator in $S$ and contains the labelled path of every reduced word $u\in U$. Whether the $u$-path returns to the root is therefore determined by this neighbourhood. Restricting each finite $F_\infty$-model to the first $|S|$ generators gives a finite $F(S)$-set. These models may be chosen so that, simultaneously for all $u\in U$,
\[
 \left|\frac{|\{z\in V:u\cdot z=z\}|}{|V|}-(\mu\times\nu)(\Fix(\rho(u)))\right|<\eta
\]
for some finite $F(S)$-set $V$.

If $P_u(x)=1_{F(S)}$, then \eqref{eq:endpoint} gives $u\cdot x=x$, so $\rho(u)$ fixes every point over $x$. Conversely, if $\rho(u)(x,y)=(x,y)$, then $P_u(x)\cdot y=y$, and freeness of the action on $Y$ forces $P_u(x)=1_{F(S)}$. Thus
\[
 (\mu\times\nu)(\Fix(\rho(u)))=\mu\{x:P_u(x)=1_{F(S)}\}.
\]
Taking $d=|V|$ and writing the finite $F(S)$-action as a homomorphism $\sigma:F(S)\to\Sym(d)$ proves the lemma.
\end{proof}

\begin{proposition}\label{prop:criterion}
Let $S$ and $R$ be finite, put $\Gamma=\langle S\mid R\rangle$, and let $q:F(S)\twoheadrightarrow\Gamma$ be the quotient map. Suppose that for every $\eps>0$ there exist a free pmp action $\Gamma\curvearrowright(X,\mu)$ and a cocycle satisfying \eqref{eq:lift} such that, with $\operatorname{err}_\alpha$ defined from this cocycle as above,
\[
 \operatorname{err}_\alpha(r)<\eps\qquad(r\in R).
\]
Then $\Gamma$ is sofic.
\end{proposition}

\begin{proof}
Fix a finite $W\subseteq F(S)$ and $\delta>0$. For each $u\in W\cap\normal{R}$ choose an integer $\ell(u)\ge0$ and an expression
\[
 u=\prod_{j=1}^{\ell(u)} c_j r_j^{\epsilon_j} c_j^{-1},
 \qquad c_j\in F(S),\quad r_j\in R,\quad \epsilon_j\in\{\pm1\}.
\]
Let
\[
 L=\max\bigl(\{\ell(u):u\in W\cap\normal{R}\}\cup\{1\}\bigr)
\]
and choose the cocycle so that $\operatorname{err}_\alpha(r)<\delta/(4L)$ for all $r\in R$. Equations \eqref{eq:error-basic} and \eqref{eq:error-conj} give
\[
 \operatorname{err}_\alpha(u)<\delta/4\qquad(u\in W\cap\normal{R}).
\]

If $u\in W\setminus\normal{R}$, then $q(u)\ne1$. Freeness of the $\Gamma$-action and \eqref{eq:endpoint} imply $P_u(x)\ne1_{F(S)}$ for every $x\in X$. Apply \cref{lem:trace-realization} to $W$ with error smaller than $\delta/4$. Words in $W\cap\normal{R}$ then have permutation trace greater than $1-\delta/2$, whereas words in $W\setminus\normal{R}$ have trace smaller than $\delta/4$. These estimates are stronger than the finite-word criterion above with tolerance $\delta$. Hence $\Gamma$ is sofic.
\end{proof}

\section{Cocycle repair from terminal replacement}\label{sec:repair}

Let $S=T\sqcup\{s\}$ and suppose that a word $u=\omega s$ acts trivially on a pmp $F(S)$-space $(X,\mu)$. Suppose that, for every $a\in T^{\pm1}$, Borel maps
\[
 p_a:X\longrightarrow F(S)
\]
have already been chosen and satisfy \eqref{eq:generator-compatibility}. For $x\in X$ write
\[
 e_x=(s\cdot x,x)
\]
for the oriented $s$-edge from $s\cdot x$ back to $x$. If $x$ belongs to a Borel set $A$ of deleted-edge basepoints, replacing $e_x$ by the complementary $u$-path gives another path with the same endpoints. We use a replacement operator $\theta$ that makes this substitution on $e_x^{\pm1}$, fixes all other oriented generator edges, respects inversion, and extends to finite paths by concatenation.

Since the rightmost letter $s$ in $u=\omega s$ is traversed first, the $u$-cycle based at $x$ first traverses $e_x^{-1}$ from $x$ to $s\cdot x$ and then follows the $\omega$-path from $s\cdot x$ back to $x$. Equivalently, $\theta(e_x)$ is that complementary $\omega$-path, oriented from $s\cdot x$ to $x$.

For a finite oriented generator path $\gamma$ containing no deleted $s$-edge, define its \emph{relative value} $\operatorname{val}_p(\gamma)\in F(S)$ as follows. Traverse $\gamma$ from its initial point. An edge labelled $a\in T^{\pm1}$ and traversed from $z$ to $a\cdot z$ contributes $p_a(z)$; an undeleted edge labelled $s^{\pm1}$ contributes the ordinary letter $s^{\pm1}$. The factors are multiplied in the same right-to-left convention as $P_u$. By \eqref{eq:generator-compatibility}, if $\gamma$ runs from $z$ to $z'$, then
\begin{equation}\label{eq:relative-path-endpoint}
 \operatorname{val}_p(\gamma)\cdot z=z'.
\end{equation}
When $p_a\equiv a$ for $a\in T^{\pm1}$, this is just the free-group word carried by the path.

\begin{lemma}\label{lem:terminal-repair}
Assume that $\theta$ is Borel and that for some $N\ge1$,
\[
 \theta^N(e_x)
\]
contains no deleted edge for every $x\in A$. Then the prescribed maps $p_a$, $a\in T^{\pm1}$, extend to Borel maps $p_s,p_{s^{-1}}$ satisfying \eqref{eq:generator-compatibility} such that
\[
 P_u(x)=1_{F(S)}\qquad(x\in A).
\]
The maps attached to $T^{\pm1}$ are unchanged.
\end{lemma}

\begin{proof}
Put
\[
 \tau_x=\theta^N(e_x)\qquad(x\in A).
\]
The path $\tau_x$ runs from $s\cdot x$ to $x$ and contains no deleted edge. Since $N$ is fixed, $\theta$ is Borel, and the prescribed maps on $T^{\pm1}$ are Borel, the map
\[
 x\longmapsto \operatorname{val}_p(\tau_x)
\]
is Borel. Define
\[
 p_s(x)=
 \begin{cases}
  \operatorname{val}_p(\tau_x)^{-1},&x\in A,\\
  s,&x\notin A,
 \end{cases}
 \qquad
 p_{s^{-1}}(s\cdot x)=p_s(x)^{-1}.
\]
Since $x\mapsto s\cdot x$ is a Borel bijection, the second formula defines $p_{s^{-1}}$ on all of $X$. By \eqref{eq:relative-path-endpoint}, for $x\in A$ the word $\operatorname{val}_p(\tau_x)^{-1}$ carries $x$ to $s\cdot x$. Hence the new maps satisfy \eqref{eq:generator-compatibility}.

Let $\gamma_x=\theta(e_x)$ be the complementary path from $s\cdot x$ to $x$. Evaluate the cocycle product $P_\omega(s\cdot x)$ using the new maps. A prescribed $T$-edge contributes its prescribed value. An undeleted $s$-edge contributes the ordinary letter. If a deleted occurrence is based at $y\in A$, then an occurrence of $e_y$ contributes
\[
 p_s(y)^{-1}=\operatorname{val}_p(\tau_y),
\]
and an occurrence of $e_y^{-1}$ contributes its inverse. Thus replacing each deleted edge in $\gamma_x$ by its terminal path $\tau_y=\theta^N(e_y)$ does not change the evaluated product. Since $\theta$ respects concatenation and $\tau_x$ contains no deleted edge,
\[
 \theta^N(\gamma_x)=\theta^{N+1}(e_x)=\theta(\tau_x)=\tau_x.
\]
In particular $\theta^N(\gamma_x)$ contains no deleted edge, so its relative value is defined, and
\[
 P_\omega(s\cdot x)
 =\operatorname{val}_p\bigl(\theta^N(\gamma_x)\bigr)
 =\operatorname{val}_p(\tau_x)
\]
and hence
\[
 P_u(x)=P_\omega(s\cdot x)p_s(x)=1_{F(S)}.
\]
\end{proof}

\section{Poulin--Wr\'obel edge cutting in a finite alphabet}\label{sec:pwinput}

We use the one-relator product edge-cutting construction of Poulin and Wr\'obel in the following form. Let $\Lambda$ be finitely generated and locally indicable, let $T$ be a finite generating set for $\Lambda$, and put
\[
 \Gamma_m=(\Lambda*\langle s\rangle)/\normal{w^m},\qquad m\ge1,
\]
where $w\in\Lambda*\langle s\rangle$ is not conjugate into either free-product factor and is not a proper power.

After cyclic conjugation and, if necessary, inversion of $w$, choose a cyclically reduced free-product normal-form representative $w_0$ so that, after spelling each $s$-syllable as a string of $s^{\pm1}$, a distinguished $s$-edge in $w_0^m$ is based at the identity and is traversed first with label $s$. For each $\Lambda$-syllable of $w_0$, choose a representing word over $T^{\pm1}$, and let $v\in F(T\sqcup\{s\})$ be the resulting expansion. With this normalization,
\[
 v^m=\omega s
\]
for some $\omega\in F(T\sqcup\{s\})$. Write $\bar v$ for the common image of $v$ and $w_0$ in $\Gamma_m$.

\begin{proposition}\label{prop:pw-terminal}
In the preceding normalization one has $|\langle\bar v\rangle|=m$. For every $\kappa>0$ there exist a free pmp action $\Gamma_m\curvearrowright(X,\mu)$, a Borel set $A\subseteq X$, a Borel replacement operator $\widehat\theta$ for the deleted $s$-edges based in $A$, and an integer $N\ge1$ such that
\[
 \mu(A)\ge \frac1m-\kappa,
 \qquad
 A\text{ is contained in a Borel transversal for }\langle\bar v\rangle,
\]
and $\widehat\theta^N(e_x)$ contains no deleted edge for every $x\in A$.
\end{proposition}

\begin{proof}
The proof of Poulin--Wr\'obel's one-relator product theorem verifies the hypotheses of their edge-cutting theorem and performs the required cyclic normalization in this setting \cite[proof of Theorem~4.1]{PoulinWrobel2026}. The proof of their Theorem~3.4 also records that the image of the root word has order exactly $m$. Their edge-cutting theorem and its proof then give, for every $\kappa>0$, a free pmp action together with a deleted-edge basepoint set of measure at least $1/m-\kappa$ contained in a transversal for the order-$m$ subgroup $\langle\bar v\rangle$ \cite[Theorem~3.4 and its proof]{PoulinWrobel2026}. We denote this deleted-edge basepoint set by $A$. If $\theta$ denotes their replacement operator in the alphabet $\Lambda\cup\{s\}$, Lemma~3.3 and Claim~3.5 in that proof provide a single integer $N$, independent of $x$, such that $\theta^N(e_x)$ contains no deleted edge for every $x\in A$ \cite[Lemma~3.3 and Claim~3.5]{PoulinWrobel2026}.

Expand every $\Lambda$-syllable occurrence in $w_0$ by its chosen word over $T^{\pm1}$. For a fixed occurrence of $s^{\pm1}$, the product of the preceding letters is unchanged as an element of $\Lambda*\langle s\rangle$; hence the corresponding $s$-edge has the same basepoint before and after expansion. The inserted letters lie in $T^{\pm1}$, so no new $s$-edge is created. It follows that expansion carries the complementary relator path for a deleted edge to the complementary path for the expanded relator and preserves the ordered deleted $s$-edge occurrences, their orientations, and their basepoints. Define $\widehat\theta$ by making the corresponding complementary-path substitution for the relator $v^m$, fixing all surviving generator edges, respecting inversion, and extending by concatenation, as in \cref{sec:repair}. The substitution is finite and fixed, so $\widehat\theta$ is Borel. By induction on $n$, expansion of $\theta^n(e_x)$ is $\widehat\theta^n(e_x)$ for every $n\ge0$. The termination bound $N$ therefore persists after expansion.
\end{proof}

\section{A relative one-step repair}\label{sec:one-step}

The preceding edge cutting can be combined with an arbitrary cocycle on the lower factor. This is the form used in the induction.

The following elementary transport lemma records the part of the construction that preserves the lower cocycle.

\begin{lemma}\label{lem:transport}
Let $T\subset S$, let $\Lambda\le\Gamma$, and suppose the chosen maps $F(T)\twoheadrightarrow\Lambda$ and $F(S)\twoheadrightarrow\Gamma$ are compatible with the inclusion $F(T)\le F(S)$. Let $\Lambda\curvearrowright(Y,\nu)$ be a free pmp action equipped with an $F(T)$-valued cocycle satisfying the lifting condition. Suppose $\Gamma\curvearrowright(Z,\eta)$ is a pmp action and $\pi:Z\to Y$ is a $\Lambda$-equivariant factor map. For $a\in T^{\pm1}$ define
\[
 \widetilde p_a(z)=p_a(\pi z),
\]
viewing $F(T)$ as the subgroup $F(T)\le F(S)$. Then the maps $\widetilde p_a$ satisfy \eqref{eq:generator-compatibility} for the restricted $\Lambda$-action on $Z$, and for every $v\in F(T)$,
\begin{equation}\label{eq:transport}
 \widetilde P_v(z)=P_v(\pi z).
\end{equation}
In particular the measure of $\{\widetilde P_v\ne1\}$ equals the measure of $\{P_v\ne1\}$.
\end{lemma}

\begin{proof}
The lifting condition on $Y$ says that $p_a(y)$ and $a$ represent the same element of $\Lambda$. Since $\Lambda\le\Gamma$, this gives
\[
 \widetilde p_a(z)\cdot z=a\cdot z.
\]
Using $\Lambda$-equivariance of $\pi$ also gives
\[
 \widetilde p_{a^{-1}}(a\cdot z)
 =p_{a^{-1}}(a\cdot\pi z)
 =p_a(\pi z)^{-1}.
\]
Equation \eqref{eq:transport} follows from the word formula for the cocycle, and the equality of measures follows from $\pi_*\eta=\nu$.
\end{proof}

If $\Lambda\le\Gamma$, the coinduced action of $\Gamma$ associated to a pmp $\Lambda$-action has a $\Lambda$-equivariant factor map onto the original action; see Kechris \cite[Section~II.10.G]{Kechris2010}. Taking a product with a free Bernoulli $\Gamma$-action if necessary makes the action free without changing this factor map.

\begin{proposition}\label{prop:relative-step}
Keep the setting and normalization preceding \cref{prop:pw-terminal}. Let $\Lambda\curvearrowright(Y,\nu)$ be a free pmp action equipped with an $F(T)$-valued cocycle satisfying the lifting condition for $F(T)\twoheadrightarrow\Lambda$, and denote its generator maps by $p_a^Y$. Then for every $\kappa>0$ there exist a free pmp action $\Gamma_m\curvearrowright(\Omega,\lambda)$, a $\Lambda$-equivariant factor map $\pi:\Omega\to Y$, and an $F(T\sqcup\{s\})$-valued cocycle satisfying the lifting condition such that
\[
 p_a=p_a^{\,Y}\circ\pi\qquad(a\in T^{\pm1})
\]
and
\[
 \operatorname{err}_\alpha(v^m)\le m\kappa.
\]
In particular, under the natural inclusion $F(T)\le F(T\sqcup\{s\})$, every cocycle value of a word in $F(T)$ has the same distribution as before the extension.
\end{proposition}

\begin{proof}
By Howie's Freiheitssatz, $\Lambda$ embeds in $\Gamma_m$ \cite{Howie1981}. Coinduce the given $\Lambda$-action to $\Gamma_m$ and, if necessary, take its product with a free Bernoulli $\Gamma_m$-action. This gives a free pmp action $\Gamma_m\curvearrowright(Z,\eta)$ with a $\Lambda$-equivariant factor map $\pi_Z:Z\to Y$. By \cref{lem:transport}, the lower generator maps pull back to $Z$ with all lower-word cocycle distributions unchanged.

Apply \cref{prop:pw-terminal} with the given $\kappa$. It gives independently a free pmp action $\Gamma_m\curvearrowright(X_0,\mu_0)$, a deleted-edge basepoint set $A_0$, a replacement operator $\widehat\theta_0$, and a uniform termination bound $N$, with
\[
 \mu_0(A_0)\ge\frac1m-\kappa
\]
and $A_0$ contained in a transversal for $\langle\bar v\rangle$.

Use the diagonal action on $\Omega=Z\times X_0$ with the product measure $\lambda=\eta\times\mu_0$, and let $\pi=\pi_Z\circ\operatorname{pr}_Z$. Pull the lower generator maps from $Z$ to $\Omega$ and put $A=Z\times A_0$. Lift $\widehat\theta_0$ to $\Omega$ by making exactly the same complementary-word substitution, with deletion determined by the second coordinate. Equivariance of the projection $\Omega\to X_0$ shows that the same bound $N$ terminates every replacement starting from $A$. Since $A_0$ is contained in a transversal for $\langle\bar v\rangle$ and the projection $\Omega\to X_0$ is equivariant, $A$ is contained in a transversal for the induced cyclic action on $\Omega$.

Apply \cref{lem:terminal-repair} to $u=v^m=\omega s$. The repair changes only the maps attached to $s^{\pm1}$, so the displayed formula for the lower maps is preserved, and
\[
 P_{v^m}(x)=1_{F(T\sqcup\{s\})}\qquad(x\in A).
\]
Let
\[
 B_v=\{x:P_{v^m}(x)=1_{F(T\sqcup\{s\})}\}.
\]
By \eqref{eq:zero-set-invariance}, $B_v$ is invariant under $v$. Since the image $\bar v$ has order $m$ and $A$ lies in a transversal for its cyclic action,
\[
 A,\ \bar v\cdot A,\ldots,\bar v^{m-1}\cdot A
\]
are pairwise disjoint subsets of $B_v$. Hence
\[
 \lambda(B_v)\ge m\lambda(A)=m\mu_0(A_0)\ge1-m\kappa,
\]
which gives $\operatorname{err}_\alpha(v^m)\le m\kappa$. The last assertion follows from the factor map and the word formula for the transported cocycle.
\end{proof}
\section{Proof of the main theorem}\label{sec:main-proof}

\begin{proof}[Proof of \cref{thm:main}]
Let
\[
 \Gamma=\langle S\mid r\rangle
\]
with $S$ arbitrary. The word $r$ involves only a finite subset $S_0\subseteq S$, and
\[
 \Gamma\cong \langle S_0\mid r\rangle * F(S\setminus S_0).
\]
Every finite subset of $\Gamma$ is contained in a subgroup
\[
 \langle S_0\mid r\rangle * F(S_1)
\]
for some finite $S_1\subseteq S\setminus S_0$. By locality of soficity and closure under free products \cite{ElekSzabo2006}, it therefore suffices to treat the case in which $S$ is finite. If $r=1$ in $F(S)$, then $\Gamma$ is free and hence sofic. Assume $r\ne1$, replace it by a cyclically reduced conjugate, and write
\[
 r=w^m,
\]
where $m\ge1$ and $w$ is not a proper power.

If $w$ involves only one generator, then after inversion $w=s$ for some $s\in S$, and
\[
 \Gamma\cong F(S\setminus\{s\})*\mathbb Z/m\mathbb Z,
\]
which is sofic by closure under free products \cite{ElekSzabo2006}. Assume henceforth that $w$ involves at least two generators.

Choose a generator $s$ occurring in $w$, put $T=S\setminus\{s\}$ and $\Lambda=F(T)$, and view $w$ in $\Lambda*\langle s\rangle$. The free group $\Lambda$ is locally indicable. Because $w$ is cyclically reduced in the basis $S$, any cyclically reduced conjugate is a cyclic permutation of $w$. Since $w$ contains both an $s^{\pm1}$-letter and a letter from $T^{\pm1}$, it is conjugate neither into $F(T)$ nor into $\langle s\rangle$; and it is not a proper power by construction. Thus the hypotheses of the normalization preceding \cref{prop:pw-terminal} are satisfied. Choose the corresponding $w_0$ and its expansion $v$. Since $\Lambda=F(T)$, the expansion $v$ represents in $F(S)=F(T)*\langle s\rangle$ the same element as $w_0$. As $w_0$ is conjugate to $w$ or $w^{-1}$, the words $v^m$ and $w^m$ have the same normal closure.

Fix $\eps>0$ and choose $\kappa>0$ with $m\kappa<\eps$. Give $\Lambda=F(T)$ any free pmp action and its canonical cocycle $p_a\equiv a$. Applying \cref{prop:relative-step} gives a lifting cocycle for the resulting $\Gamma$-action with
\[
 \operatorname{err}_\alpha(v^m)\le m\kappa<\eps.
\]
Since $\eps$ was arbitrary, \cref{prop:criterion} applied to the presentation $\langle S\mid v^m\rangle$ proves that $\Gamma$ is sofic.
\end{proof}

\section{Iteration over reducible presentations}\label{sec:reducible-proof}

\begin{proposition}\label{prop:induction}
Let $\Gamma=\langle S\mid R\rangle$ be a finite reducible presentation with no proper powers. For every $\eps>0$ there are a free pmp action $\Gamma\curvearrowright(X,\mu)$ and a cocycle satisfying \eqref{eq:lift} such that
\[
 \operatorname{err}_\alpha(r)<\eps\qquad(r\in R).
\]
\end{proposition}

\begin{proof}
We induct on $|R|$. If $R=\varnothing$, take any free pmp action and the canonical cocycle $P_v(x)=v$.

Assume $R\ne\varnothing$. Reducibility gives $s\in S$ and $r\in R$ such that, with
\[
 S'=S\setminus\{s\},\qquad R'=R\setminus\{r\},\qquad
 \Lambda=\langle S'\mid R'\rangle,
\]
the lower presentation is reducible and the image of $r$ in $\Lambda*\langle s\rangle$ is not conjugate into $\Lambda$. The no-proper-power condition passes to the lower presentation and says that the image of $r$ is not a proper power in $\Lambda*\langle s\rangle$. The lower group $\Lambda$ is locally indicable by Howie's theorem \cite[Corollary~4.5]{Howie1982}.

Let
\[
 \pi_0:F(S)\longrightarrow\Lambda*\langle s\rangle
\]
be the quotient obtained by imposing the lower relations $R'$.

If $\pi_0(r)$ is conjugate into $\langle s\rangle$, then, since it is not a proper power, it is conjugate to $s$ or $s^{-1}$. The final group is isomorphic to $\Lambda$ and $s$ acts trivially. Choose $g\in\Lambda*\langle s\rangle$ and $\zeta\in\{\pm1\}$ with
\[
 s=g\pi_0(r)^\zeta g^{-1},
\]
and choose $c\in F(S)$ with $\pi_0(c)=g$. Then
\[
 k_0=r c^{-1}s^{-\zeta}c\in\normal{R'}.
\]
Fix an expression
\[
 k_0=\prod_{j=1}^{M_0}h_j(r_j')^{\delta_j}h_j^{-1},
 \qquad r_j'\in R',\quad \delta_j\in\{\pm1\}.
\]
Apply the induction hypothesis to the lower presentation with tolerance $\xi>0$ such that $(M_0+1)\xi<\eps$. Regard the resulting action as an action of $\Gamma\cong\Lambda$, embed the lower cocycle values in $F(S)$, and put
\[
 p_s(x)=p_{s^{-1}}(x)=1_{F(S)}.
\]
The lower relator errors are unchanged and $\operatorname{err}_\alpha(s)=0$. By \eqref{eq:error-basic} and \eqref{eq:error-conj},
\[
 \operatorname{err}_\alpha(r)
 \le \sum_{j=1}^{M_0}\operatorname{err}_\alpha(r_j')
 \le M_0\xi<\eps.
\]
This proves the induction step in this case.

Assume now that $\pi_0(r)$ is conjugate into neither free-product factor. Use the normalization preceding \cref{prop:pw-terminal}, with $m=1$, for the conjugacy class of $\pi_0(r)$ or its inverse. Let $v\in F(S)$ be the resulting finite-alphabet word, so that
\[
 v=\omega s.
\]
There are $g\in\Lambda*\langle s\rangle$ and $\zeta\in\{\pm1\}$ such that
\begin{equation}\label{eq:normalized-top-element}
 \pi_0(v)=g\pi_0(r)^\zeta g^{-1}.
\end{equation}
Choose $c\in F(S)$ with $\pi_0(c)=g$ and put
\begin{equation}\label{eq:lower-kernel-word}
 k=r c^{-1}v^{-\zeta}c.
\end{equation}
Then $k\in\ker\pi_0=\normal{R'}$. Fix an expression
\begin{equation}\label{eq:lower-kernel-expression}
 k=\prod_{j=1}^{M}h_j(r_j')^{\delta_j}h_j^{-1},
 \qquad r_j'\in R',\quad \delta_j\in\{\pm1\}.
\end{equation}

Fix $\eps>0$ and choose
\[
 0<\xi<\frac{\eps}{2(M+1)}.
\]
By induction, there are a free pmp action $\Lambda\curvearrowright(Y,\nu)$ and a lower cocycle with
\begin{equation}\label{eq:lower-errors}
 \operatorname{err}_\alpha(r')<\xi\qquad(r'\in R').
\end{equation}
Choose $\kappa>0$ with $\kappa<\eps/2$. Apply \cref{prop:relative-step} to the lower action and cocycle, with $m=1$. The resulting cocycle on a free pmp $\Gamma$-action preserves all estimates \eqref{eq:lower-errors} and satisfies
\begin{equation}\label{eq:v-error}
 \operatorname{err}_\alpha(v)\le\kappa<\frac{\eps}{2}.
\end{equation}

From \eqref{eq:lower-kernel-word},
\[
 r=k c^{-1}v^\zeta c.
\]
The words $k$ and $v$ act trivially in the final group. Thus \eqref{eq:error-basic}, \eqref{eq:error-conj}, \eqref{eq:lower-kernel-expression}, and \eqref{eq:v-error} give
\[
 \operatorname{err}_\alpha(r)
 \le \operatorname{err}_\alpha(k)+\operatorname{err}_\alpha(v)
 \le \sum_{j=1}^{M}\operatorname{err}_\alpha(r_j')+\operatorname{err}_\alpha(v)
 <M\xi+\frac{\eps}{2}<\eps.
\]
This completes the induction.
\end{proof}

For a finite reducible presentation, \cref{prop:induction} and \cref{prop:criterion} prove soficity.

\subsection*{A final power relation}

\begin{proof}[Proof of \cref{thm:final-power}]
Put
\[
 \Gamma=(\Lambda*\langle s\rangle)/\normal{u^m}.
\]
The group $\Lambda$ is locally indicable by Howie's theorem \cite[Corollary~4.5]{Howie1982}. If the image of $u$ is conjugate into $\langle s\rangle$, then the non-proper-power assumption makes it conjugate to $s$ or $s^{-1}$, and
\[
 \Gamma\cong\Lambda*\mathbb Z/m\mathbb Z,
\]
where $\mathbb Z/1\mathbb Z$ is trivial. This group is sofic because $\Lambda$ is sofic by the finite reducible case and soficity is preserved under free products \cite{ElekSzabo2006}.

Assume that $u$ is conjugate into neither free-product factor. Let
\[
 \pi_0:F(S'\sqcup\{s\})\longrightarrow\Lambda*\langle s\rangle
\]
be the quotient by the lower relations. Use the normalization preceding \cref{prop:pw-terminal} for the conjugacy class of the image of $u$ or its inverse. Let $v\in F(S'\sqcup\{s\})$ be the resulting finite-alphabet word, so that
\[
 v^m=\omega s.
\]
There are $g\in\Lambda*\langle s\rangle$ and $\zeta\in\{\pm1\}$ with
\[
 \pi_0(v)=g\pi_0(u)^\zeta g^{-1}.
\]
Choose $c\in F(S'\sqcup\{s\})$ with $\pi_0(c)=g$. Then
\[
 k=u^m c^{-1}(v^m)^{-\zeta}c\in\normal{R'}.
\]
Fix an expression of $k$ as a product of $M$ conjugates of elements of $R'^{\pm1}$.

Fix $\eps>0$. By \cref{prop:induction}, choose a lower cocycle with
\[
 \operatorname{err}_\alpha(r')<\xi\qquad(r'\in R')
\]
for $\xi>0$ so small that $M\xi<\eps/2$. Choose $\kappa>0$ with $m\kappa<\eps/2$. Apply \cref{prop:relative-step} to the lower action and cocycle. The resulting cocycle on a free pmp $\Gamma$-action preserves all lower relator errors and satisfies
\[
 \operatorname{err}_\alpha(v^m)\le m\kappa<\frac\eps2.
\]
Since
\[
 u^m=k c^{-1}(v^m)^\zeta c,
\]
the error inequalities give
\[
 \operatorname{err}_\alpha(u^m)
 \le M\xi+\operatorname{err}_\alpha(v^m)<\eps.
\]
The cocycle criterion proves that $\Gamma$ is sofic.
\end{proof}

\section{Arbitrary reducible presentations}\label{sec:consequences}

\begin{lemma}\label{lem:finite-subpresentation-test}
Let $\Gamma=\langle S\mid R\rangle$ be a presentation, with $S$ and $R$ arbitrary, and let $W\subset F(S)$ be finite. Then there is a finite subpresentation
\[
 \Gamma_0=\langle S_0\mid R_0\rangle
\]
with $S_0\subset S$ and $R_0\subset R$ such that every word of $W$ is a word over $S_0$ and
\[
 u=1\text{ in }\Gamma_0\quad\Longleftrightarrow\quad u=1\text{ in }\Gamma
 \qquad(u\in W).
\]
\end{lemma}

\begin{proof}
Put into $S_0$ every generator occurring in a word of $W$. For each $u\in W$ that is trivial in $\Gamma$, choose a finite expression
\[
 u=\prod_{j=1}^{m(u)}h_{u,j}r_{u,j}^{\epsilon_{u,j}}h_{u,j}^{-1},
 \qquad r_{u,j}\in R,\quad \epsilon_{u,j}\in\{\pm1\}.
\]
Add the finitely many relators $r_{u,j}$ to $R_0$ and enlarge $S_0$ to contain every generator occurring in them and in the conjugating words. Then every word of $W$ that is trivial in $\Gamma$ is already trivial in $\Gamma_0$. Conversely, the inclusion of generators and relators gives a homomorphism $\Gamma_0\to\Gamma$, so a word trivial in $\Gamma_0$ is trivial in $\Gamma$.
\end{proof}

\begin{proof}[Proof of \cref{thm:reducible}]
The finite case follows from \cref{prop:induction} and \cref{prop:criterion}. Let $\Gamma=\langle S\mid R\rangle$ have a reducible presentation without proper powers, with $S$ and $R$ arbitrary. Fix a finite $K\subset\Gamma$ and $\delta>0$, enlarge $K$ to contain $1$, and choose words $\widehat g\in F(S)$ representing every element of $K\cup K^2$, with $\widehat1=1$. Let $W$ consist of
\[
 \widehat g\widehat h\widehat{gh}^{-1}\qquad(g,h\in K)
\]
and
\[
 \widehat g\widehat h^{-1}\qquad(g,h\in K,\ g\ne h).
\]
Apply \cref{lem:finite-subpresentation-test}. In the resulting finite subpresentation $\Gamma_0$, the first family is trivial and the second family is nontrivial. Thus the elements represented by the $\widehat g$ preserve the products from $K$ exactly and are distinct. The finite subpresentation is reducible and has no proper powers, so it is sofic by the finite case. A sufficiently good finite sofic approximation of these finitely many elements gives the approximation required in the definition for $K$ and $\delta$.
\end{proof}

\begin{corollary}\label{cor:arbitrary-final-power}
Let $\Lambda=\langle S'\mid R'\rangle$ be given by a reducible presentation without proper powers. Let $s\notin S'$ and let $u\in F(S'\sqcup\{s\})$ have image in $\Lambda*\langle s\rangle$ which is not conjugate into $\Lambda$ and is not a proper power. Then, for every $m\ge1$,
\[
 (\Lambda*\langle s\rangle)/\normal{u^m}
\]
is sofic.
\end{corollary}

\begin{proof}
Fix a finite soficity test in the displayed group and apply \cref{lem:finite-subpresentation-test}. Enlarge the resulting finite subpresentation to contain the relation $u^m$ and all generators occurring in $u$. The truth values of the finitely many test words are unchanged by this enlargement: words already trivial remain trivial, while a word nontrivial in the original group cannot become trivial in the enlarged subpresentation because the latter still maps to the original group. Its lower part is a finite subpresentation of the reducible presentation of $\Lambda$, hence is reducible and has no proper powers. If the image of $u$ in the corresponding finite lower free product were conjugate into the lower group, or were a proper power, its image in $\Lambda*\langle s\rangle$ would have the same property. Thus \cref{thm:final-power} applies to the finite subpresentation. A sufficiently good approximation of the chosen finite test transfers to the original group.
\end{proof}

\subsection*{Acknowledgements}
I thank Antoine Poulin and Konrad Wr\'obel for helpful comments on earlier versions, particularly Poulin's observation concerning the cocycle interpretation and Wr\'obel's suggestion to consult the work of Elek and Lippner.

\subsection*{Tool and AI use disclosure}
AI tools, including large language models, were used during exploratory work and manuscript preparation for English drafting and translation, literature and citation checks, and searches for errors or gaps in candidate arguments (including checks for possible gaps in the use of the Elek--Lippner results in \cref{lem:trace-realization}). The author checked the mathematical arguments, citations, and attribution and takes responsibility for the content of the paper.

\end{document}